\documentclass[a4paper]{amsart}
\usepackage{amsmath,amssymb, amsthm, mathrsfs}
\usepackage{url}
\usepackage[all]{xy}
\usepackage[shortlabels]{enumitem}
\usepackage{hyperref}

\numberwithin{equation}{section}

\newcommand{\C}{{\mathcal C}}
\newcommand{\Cu}{\mathrm{Map}_{{\C}}}
\newcommand{\D}{{\mathcal D}}
\newcommand{\Du}{\mathrm{Map}_{{\D}}}
\newcommand{\E}{{\mathcal E}}
\newcommand{\Eu}{\mathrm{Map}_{{\E}}}
\newcommand{\CD}{\C^{\D}}
\newcommand{\CDu}{\mathrm{Map}_{\CD}}
\newcommand{\F}{{\mathcal F}}
\newcommand{\U}{{\mathcal U}}
\newcommand{\Uu}{\mathrm{Map}_{{\U}}}
\newcommand{\UG}{{{\mathcal U}^G}}
\newcommand{\UD}{{{\mathcal U}^{\mathcal D}}}
\newcommand{\UGu}{\mathrm{Map}_{{\mathcal U}^G}}

\newcommand{\OF}{{\mathcal O_{\mathcal F}}}
\newcommand{\OFop}{{{\mathcal O}_{\mathcal F}^\mathrm{op}}}
\newcommand{\OFopu}{\mathrm{Map}_{{\mathcal O}_{\mathcal F}^\mathrm{op}}}
\newcommand{\COFop}{{\C^{{\mathcal O}_{\mathcal F}^\mathrm{op}}}}
\newcommand{\id}{\mathrm{id}}
\newcommand{\op}{\mathrm{op}}
\newcommand{\ev}{\mathrm{ev}}
\DeclareMathOperator{\colim}{colim}
\newcommand{\Set}{\textbf{Set}}
\newcommand{\sSet}{\textbf{sSet}}
\newcommand{\Gr}{\textbf{Gr}}
\newcommand{\sGr}{\textbf{sGr}}
\newcommand{\Cat}{\textbf{Cat}}
\newcommand{\Pos}{\textbf{Pos}}
\newcommand{\RMod}{{_R}\textbf{Mod}}
\newcommand{\Fi}{F_\iota}
\newcommand{\Ui}{U_\iota}
\newcommand{\Ho}{\mathrm{Ho}}
\newcommand{\Map}{\mathrm{Map}}
\newcommand{\Path}{\mathrm{Path}}
\newcommand{\Sk}{\mathrm{Sk}}
\newcommand{\Ch}{\mathrm{Ch}}

\newtheoremstyle{slanted} 
{} 
{} 
{\slshape} 
{} 
{\bfseries} 
{.} 
{.5em} 
{} 

\makeatletter
\theoremstyle{slanted}
\newtheorem*{rep@theorem}{\rep@title}
\newcommand{\newreptheorem}[2]{%
\newenvironment{rep#1}[1]{%
 \def\rep@title{#2 \ref{##1}}%
 \begin{rep@theorem}}%
 {\end{rep@theorem}}}
\makeatother

\theoremstyle{definition}

\newtheorem{defn}{Definition}[section]
\newtheorem{rem}[defn]{Remark}
\newtheorem{ex}[defn]{Example}
\theoremstyle{slanted}
\newtheorem{cor}[defn]{Corollary}
\newtheorem{lem}[defn]{Lemma}
\newtheorem{thm}[defn]{Theorem}
\newreptheorem{thm}{Theorem}
\newtheorem{prop}[defn]{Proposition}

\begin{document}
\title{On equivariant homotopy theory for model categories}
\author{Marc Stephan}

\address{Department of Mathematics, The University of Chicago, Chicago, IL 60637} 
\email{mstephan@math.uchicago.edu}
\date{\today}

\begin{abstract}
We introduce and compare two approaches to equivariant homotopy theory in a topological or ordinary Quillen model category.  For the topological model category of spaces, we generalize Piacenza's result that the categories of topological presheaves indexed by the orbit category of a fixed topological group $G$ and the category of $G$-spaces can be endowed with Quillen equivalent model category structures.  We prove an analogous result for any cofibrantly generated model category and discrete group 
$G$, under certain conditions on the fixed point functors of the subgroups of $G$.  These conditions hold in many examples, though not in the category of chain complexes, where we nevertheless establish and generalize to collections an equivariant Whitehead Theorem \`{a} la Kropholler and Wall for the normalized chain complexes of simplicial $G$-sets.
\end{abstract}

\maketitle{}

\section{Introduction}
\label{sec:Introduction}
For any topological group $G$, Piacenza (\cite[VI.~ \S6]{mayequi}, \cite{piacenza}) showed that the category of $G$-spaces and the category of orbit diagrams form Quillen equivalent model categories. This works more generally for any collection $\mathcal F$ of closed subgroups of $G$ that contains the trivial subgroup. Explicitly, the category of continuous contravariant diagrams of spaces indexed by the full subcategory $\OF$ of the orbit category with orbit spaces $G/H$ for $H\in \mathcal F$ equipped with the projective model structure is Quillen equivalent to the category of $G$-spaces with the $\mathcal F$-model structure, where the weak equivalences and fibrations are the maps that are taken to weak equivalences and fibrations by each $H$-fixed point functor for $H\in \mathcal F$.

Replacing the category of spaces with an arbitrary topological model category $\C$ or working with a discrete group $G$ and any model category $\C$, we explore the following questions. 
\begin{enumerate}
\item
Does the category $\C^G$ of $G$-objects in $\C$ admit the $\F$-model structure?
\item
Does the category of orbit diagrams $\COFop$ admit the projective model structure?
\item
If so, are $\C^G$ and $\COFop$ Quillen equivalent model categories?
\end{enumerate}

\subsection*{Results in the discrete setting}
Our main result in the discrete setting (Theorem \ref{mainthm1}) provides a positive answer to these three questions for cofibrantly generated $\C$ when the $H$-fixed point functors satisfy the cellularity conditions stated in Proposition \ref{propFmodelstructure}. The following examples are contained in \S\ref{sec:positivieexamples}.
\begin{ex}
The cellularity conditions hold and a positive answer to the three aforementioned questions is obtained if $\C$ is 
\begin{enumerate}
\item
the category of simplicial sets with the Quillen model structure, cubical sets or any other presheaf category with a cofibrantly generated model structure such that the generating cofibrations are monomorphisms,
\item 
the category of simplicial groups,
\item
the category of small categories with the Thomason model structure by work of Bohmann et aliae \cite{bohmannetal},
\item
the category of posets with Raptis's model structure \cite{raptis} by joint work with May and Zakharevich \cite{maystephanzakharevich},
\item
any cofibrantly generated left Bousfield localization of the above examples or any category of diagrams in the above examples with the projective model structure,
\item
any of the following models for $(\infty,1)$-categories, the model of complete Segal spaces, quasi-categories, simplicial categories or the two models of Segal categories by work of Bergner \cite{bergnerequi}.\footnote{Her results can be generalized from finite to discrete groups using \cite[2.3]{bohmannetal} for the model of simplicial categories, and that taking fixed points of $G
$-sets preserves directed colimits of diagrams where each arrow is a monomorphism for the models of Segal categories.}
\end{enumerate}
\end{ex}

In all the examples above, the right Quillen equivalence $\C^G\to \COFop$ takes a $G$-object to its fixed point diagram. If any of the first two questions above admits a positive answer for a model category $\C$ and discrete group $G$, then the answer to the corresponding question for the category $\C_*$ of pointed objects in $\C$ is positive as well, and if the fixed point diagram functor is a Quillen equivalence for $\C$, then the fixed point diagram functor for $\C_*$ is a Quillen equivalence as well (Lemma \ref{pointed}).  

In \S\ref{sec:chaincomplexes}, we provide an example, where the fixed point functors fail to satisfy the cellularity conditions and the fixed point diagram functor is not a Quillen equivalence.
\begin{ex}
If $\C$ is the category of non-negatively graded chain complexes with the projective model structure, then the model structures on the category of orbit diagrams $\C^\OFop$ and the  category of $G$-objects $\C^G$ exist, but the fixed point diagram functor is not a Quillen equivalence in general.
\end{ex}

One motivation for establishing $\mathcal F$-model structures for collections of subgroups $\mathcal F$ is to obtain  equivariant Whitehead Theorems for collections \cite[1.6]{lueck}. With an equivariant Whitehead Theorem we mean loosely that a map $f\colon X\to Y$ between $G$-objects is a $G$-homotopy equivalence provided that the map $f^H$ on $H$-fixed point objects is a weak equivalence for all subgroups $H$ of $G$. In the form for collections, we only need to check that $f^H$ is a weak equivalence for those subgroups $H$ that appear as isotropy groups of $X$ and $Y$.

Kropholler and Wall proved an equivariant Whitehead Theorem \cite[1.2]{krophollerwall} on chain complexes with several applications. It states that for any discrete group $G$ and ring $R$, an equivariant map $f\colon X\to Y$ between $G$-CW complexes induces an equivariant homotopy equivalence $C(f;R)$ on augmented cellular chain complexes with coefficients in $R$ provided that $C(f^H;R)\colon C(X^H;R)\to C(Y^H;R)$ is a homotopy equivalence for all subgroups $H$ of $G$. Kropholler and Wall's proof also works for the unaugmented cellular chain complexes and then implies a simplicial set version in which CW complexes are replaced by simplicial sets and the cellular chain complexes by normalized chain complexes. We provide a conceptual proof of this simplicial set version of the equivariant Whitehead Theorem and generalize it to collections.
\begin{repthm}{whiteheadchain}
Let $R$ be a ring and $\mathcal F$ a collection of subgroups of $G$. Let $f\colon X\to Y$ be a map between simplicial $G$-sets such that all isotropy groups of simplices of $X$ and $Y$ belong to $\mathcal F$. The map $C(f;R)\colon C(X;R)\to C(Y;R)$ between normalized chain complexes is an equivariant homotopy equivalence if \ref{whiteheadchaina} the induced map $C(X;R)^H\to C(Y;R)^H$ is a quasi-isomorphism for all $H\in\mathcal F$, or \ref{whiteheadchainb} the induced map $C(X^H;R)\to C(Y^H;R)$ is a quasi-isomorphism for all $H\in \mathcal F$.
\end{repthm}

\subsection*{Results in the topological setting}
For topological model categories $\C$, we prove the analogue of Theorem \ref{mainthm1} for compact Lie groups $G$.

\begin{repthm}{mainthm2}
Suppose that $G$ is a compact Lie group and that $\C$ is cofibrantly generated. Let $\mathcal F$ be a collection of closed subgroups of $G$ containing the trivial subgroup such that for any $H\in \F$, the $H$-fixed point functor satisfies the cellularity conditions \ref{cellular}. Then there is a Quillen equivalence $\COFop \leftrightarrows \C^G$
between the category of contravariant orbit diagrams with the projective model structure and of $G$-objects with the $\F$-model structure.
\end{repthm}
The cellularity conditions are often not required to establish the model structures. We show that if $G$ is a compact Lie group and $\C$ is cofibrantly generated, then $\COFop$ admits the projective model structure. If in addition every object of $\C$ is fibrant and the cofibrations are monomorphisms, then the category of $G$-objects $\C^G$ admits the $\F$-model structure (Proposition~\ref{Fmodeltop}~\ref{Fmodel_assumption2}).

For arbitrary topological groups $G$ and cofibrantly generated topological model categories $\C$ such that the cellularity conditions hold, some extra work is required to establish desired cofibrantly generated model structures on $\COFop$ and $\C^G$, but the cellularity conditions ensure that the model structures will be Quillen equivalent again. As an example, we extend the motivating result that for any topological group $G$, the categories of orbit diagrams of spaces and of $G$-spaces are Quillen equivalent, from spaces to diagrams of spaces with the projective model structure.

\subsection*{Model category theoretic techniques}
The model structures are obtained by transfer along not just one left adjoint but along a set of adjoints (Theorem \ref{transport}), as outlined in Appendix \ref{sec:transfer}. There, we summarize also the terminology about cofibrantly generated model categories (\cite{hovey}, \cite{hirschhorn}) and prove a general version of Quillen's path object argument (Lemma~\ref{lemtransport2}).

\subsection*{Context and related work}
Families or collections of subgroups are used in formulating isomorphism conjectures for $K$- and $L$-theory \cite{lueck} and the study of subgroup complexes \cite{quillensubgroup}. 

A different approach to equivariant homotopy theories was developed by Dwyer and Kan \cite{dwyerkansingular}. 

Kropholler and Wall's proof of their theorem \cite[1.2]{krophollerwall} is algebraic and works more generally \cite[1.3]{krophollerwall} for any equivariant chain map $f$ between augmented chain complexes of based permutation modules such that for every subgroup $H$ of $G$, the differentials and $f$ restrict to maps between the free modules on the $H$-fixed points of the bases. Their proof can be modified to generalize this algebraic version also to collections. In \cite{hambletonyalcin}, Hambleton and Yal{\c{c}}{\i}n gave a different, homological algebra proof of the algebraic version of the equivariant Whitehead Theorem via the orbit category.

The main results in this paper were proved in the author's master's thesis \cite{elmendorfstheorem} from 2010. May and Guillou have independently investigated equivariant homotopy theory for more general enrichments as well \cite{guilloumay1}, \cite{guilloumay2}.

\subsection*{Acknowledgments}
I am grateful to Jesper Grodal for supervising my master's thesis written as an exchange student at the University of Copenhagen and to Karin Baur, my supervisor at ETH Zurich. I would like to thank Bill Dwyer for a helpful discussion, Peter May for historical comments, Ang\'{e}lica Osorno for encouraging me to publish the obtained results and sharing the draft \cite{bohmannetal} with me, which led to improvements of the examples, Julie Bergner for discussions about her work \cite{bergnerequi}, J\'{e}r\^{o}me Scherer and Kathryn Hess for their careful reading of this paper, and the referee whose comments led to improvements of the exposition.

\section{The discrete setting}
\label{sec:discrete}
This section contains a modest extension of ideas in unpublished notes \cite{guillou} of Guillou. 

Let $G$ be a group and $\C$ a Quillen  model category \cite{quillen}.  

After defining the category of $G$-objects $\C^G$ in $\C$ and fixed point functors, we want to equip $\C^G$ with the $\mathcal F$-model structure, where the weak equivalences and fibrations are the maps that are taken to weak equivalences and fibrations, by each $H$-fixed point functor for $H$ in a fixed collection $\mathcal F$ of subgroups of $G$. The strategy is to construct left adjoints to the $H$-fixed point functors and then to apply Transfer Theorem \ref{transport}. The second approach to equivariant homotopy theory for $\C$ is to equip the category of orbit diagrams $\COFop$ with the projective model structure. We then compare the two approaches. They turn out to be Quillen equivalent in a variety of examples.   

\subsection{$G$-objects}
Identify the group $G$ with the category with one object $\ast$ and set of morphisms $G$. 
\begin{defn}
The category of \emph{$G$-objects in $\C$} is the category $\C^G$ of functors from $G$ to $\C$. For a subgroup $H$ of $G$, the \emph{$H$-fixed point functor $(-)^H$} is defined as the composition
\[
\C^G\rightarrow \C^H \stackrel{\lim}{\longrightarrow} \C
\]
of the restriction functor with the limit functor.
\end{defn}
As long as the required limit exists, we can of course use the same definition in any category.

Denote the category of sets by $\Set$.
\begin{ex}
The category of $G$-objects in $\Set$ coincides with the category of $G$-sets, i.e., of sets with a $G$-action. Moreover for any subgroup $H$ of $G$, the $H$-fixed point functor takes a $G$-set $X$ to its ordinary fixed point set 
\[X^H=\{x\in X| hx=x \text{ for all }h\in H\}.
\]
\end{ex}

We are interested in the existence of the following model category structure.
\begin{defn}
Let $\F$ be a collection of subgroups of $G$, i.e., a set of subgroups that is closed under conjugation. The category $\C^G$ is said to admit the \emph{$\F$-model structure} if it is a model category with weak equivalences and fibrations the maps that are taken to weak equivalences and fibrations in $\C$ by each $H$-fixed point functor with $H\in \F$.
\end{defn}

The strategy to equip $\C^G$ with the $\F$-model structure is to assume that $\C$ is cofibrantly generated and to transfer its model structure along the left adjoints of the fixed point functors $(-)^H$ with $H\in \F$. Abstractly, the left adjoint of $(-)^H$ is the composition
\[\C\to\C^H\to \C^G
\]
of the constant diagram functor and the induction functor, i.e., the left Kan extension functor of the inclusion functor $H\to G$. An explicit description involves the orbit set $G/H$ as a counterpart. Consider $G/H$ as a $G$-set with action of $g\in G$ on a coset $g'H$ defined by $(gg')H$. It is well-known \cite{bredonequi} that $G/H$ represents the $H$-fixed point functor $(-)^H\colon \Set^G\to \Set$.
\begin{lem}
\label{fixedpointrepresentable}
Evaluating a $G$-map $G/H\to X$ in the coset $H$ yields an isomorphism 
\[\Set^G(G/H,X)\cong X^H
\]
that is natural in the $G$-set $X$. 
\end{lem}

Before the explicit description of the left adjoint of $(-)^H$, recall that $\C$ is cocomplete and complete and thus tensored and cotensored over $\Set$. Indeed, we have isomorphisms \[\C(X\otimes A, B)\cong \Set(X,\C(A,B))\cong \C(A,X\pitchfork B)
\] 
that are natural in the set $X$ and objects $A$, $B$ of $\C$ where the tensor $X\otimes A$ is the copower $\coprod_X A$ and the cotensor $X\pitchfork B$ is the power $\prod_X B$.

For any object $A$ of $\C$ and homogeneous $G$-set $G/H$, denote the composition
\[G\stackrel{G/H}{\longrightarrow}\Set\stackrel{-\otimes A}{\longrightarrow} \C
\]
by $G/H\otimes A$.

\begin{lem}
The functor $G/H\otimes -\colon \C\to \C^G$ is left adjoint to the $H$-fixed point functor.
\end{lem}
\begin{proof}
The adjunction isomorphism is the composite
\begin{equation*}
\begin{split}
\C^G(G/H\otimes A,X)& \cong \Set^G(G/H,\C(A,X)) \\
& \cong \C(A,X)^H \\
& \cong \C(A,X^H),
\end{split}
\end{equation*}
where the first isomorphism is induced by the isomorphism expressing that $\C$ is tensored, the second isomorphism is given by Lemma~\ref{fixedpointrepresentable} and the third one comes from the fact that $\C(A,-)$ preserves limits.
\end{proof}

For cofibrantly generated $\C$, we can transfer the model structure from $\C$ to the category of $G$-objects in $\C$ under a condition on the fixed point functors. The condition is motivated by the following fact from equivariant topology. Recall that a $G$-CW complex (\cite[p.~98]{tomdiecktransf}, \cite[p.~13]{mayequi}) is a $G$-space that is obtained by attaching equivariant disks $G/K\times D^n$ along their boundaries $G/K\times S^{n-1}$. Let $H$ be a subgroup of the (discrete) group $G$. Applying the $H$-fixed point functor to a $G$-CW complex $X$ yields a CW complex $X^H$ as the $H$-fixed point functor preserves the colimits involved in the building process of $X$ and takes an equivariant disk $G/K\times D^n$ to the disjoint union of disks $(G/K)^H\times D^n$.
\begin{prop}
\label{propFmodelstructure}
Let $G$ be a group and $\F$ a collection of subgroups of $G$. Suppose that for any $H\in \F$, the $H$-fixed point functor satisfies the following \emph{cellularity conditions}:
\begin{enumerate}[i)]
\label{cellular}
\item
\label{cellular1}
$(-)^H$ preserves directed colimits of diagrams in $\C^G$, where each underlying arrow in $\C$ is a cofibration,
\item
\label{cellular2}
$(-)^H$ preserves pushouts of diagrams, where one leg is of the form
\[
G/K\otimes f\colon G/K\otimes A\to G/K\otimes B,
\]
for $K\in \F$ and $f$ a cofibration in $\C$, and
\item
\label{cellular3}
for any $K\in \F$ and object $A$ of $\C$, the induced map 
\[(G/K)^H\otimes A\to (G/K\otimes A)^H\] 
is an isomorphism in $\C$.
\end{enumerate}
If $\C$ is cofibrantly generated, then the category of $G$-objects $\C^G$ admits the $\F$-model structure and is cofibrantly generated.
\end{prop}
\begin{proof}
Denote the set of generating cofibrations of $\C$ by $I$ and of generating acyclic cofibrations by $J$. We apply the Transfer Theorem \ref{transport} to the adjunctions $\{G/H\otimes -, (-)^H\}_{H\in\mathcal F}$. As $\C$ is complete and cocomplete, so is the functor category $\C^G$. To check conditions \ref{cellular1} and \ref{cellular2}, we use Lemma~\ref{lemtransport1}, which in turn applies by the cellularity conditions. Indeed, writing $I_{\mathcal F}=\{G/H\otimes f| f\in I, H\in \mathcal F\}$, note that the underlying map in $\C$ of a relative $I_{\mathcal F}$-cell complex is a transfinite composition of pushouts of coproducts $\coprod_{G/K} f$ of generating cofibrations of $\C$, in particular a transfinite composition of cofibrations and therefore a cofibration itself. Thus by the cellularity conditions for $H\in \mathcal F$, the $H$-fixed point functor takes a relative $I_{\mathcal F}$-cell complex to a transfinite composition of pushouts of coproducts $\coprod_{(G/K)^H}f$ of generating cofibrations, which proves \ref{lemtransport1a} of Lemma~\ref{lemtransport1}. Condition \ref{lemtransport1c} holds by the cellularity condition \ref{cellular1}. One shows \ref{lemtransport1b} and \ref{lemtransport1d} similarly.       
\end{proof}

\begin{rem}
\label{rem:cellular}
Instead of the cellularity condition \ref{cellular3}, we could have required the composite $(G/K\otimes -)^H$ to take generating cofibrations to cofibrations and generating acyclic cofibrations to acyclic cofibrations. The form stated in the proposition is crucial for comparing the $\mathcal F$-model structure with the orbit diagrams in $\C$. 
\end{rem}
 A toy example, where condition \ref{cellular3} does not hold for $H=G=C_2$ the cyclic group of order two and $K$ the trivial subgroup, is the object $A=1$ in the category $0\to 1$ of two objects and one non-identity morphism as depicted.

Examples of cofibrantly generated model categories where the cellularity conditions are satisfied for any $\mathcal F$ are given in \S\ref{sec:positivieexamples}. The checking of the cellularity condition \ref{cellular2} can be reduced to generating cofibrations $f$ as in \cite{bohmannetal}.

\begin{prop}
\label{prop:cellular2}
Let $\mathcal F$ be a collection of subgroups of $G$. Let $\C$ be a cofibrantly generated model category. Suppose that for any $H\in \mathcal F$, the $H$-fixed point functor satisfies the cellularity condition \ref{cellular1} and the cellularity condition \ref{cellular2} for generating cofibrations $f\colon A\to B$. Then the $H$-fixed point functor satisfies the cellularity condition \ref{cellular2} for all cofibrations $f\colon A\to B$.
\end{prop}

\subsection{Orbit diagrams and comparison of the two approaches}
The second approach to equivariant homotopy theory for the model category $\C$ is via contravariant orbit diagrams with levelwise weak equivalences and fibrations, i.e., with the \emph{projective model structure}. As for any functor category, this works for cofibrantly generated $\C$. The indexing category is the following. 

Let $\mathcal F$ be a collection of subgroups of $G$. The \emph{orbit category $\OF$ of $G$ with respect to $\mathcal F$} is the full subcategory of the category $\Set^G$ of $G$-sets given by the orbit sets $G/H$ with $H\in \mathcal F$.

The maps $G/H\to G/K$ of the orbit category are described by Lemma~\ref{fixedpointrepresentable}. For $a\in G$ such that the coset $aK$ is in $(G/K)^H$, i.e., such that $a^{-1}Ha\subset K$, we denote the corresponding morphism $G/H \to G/K$ by $R_a$. It sends $gH$ to $gaK$.

We are ready to compare the two approaches. If $\mathcal F$ contains the trivial subgroup $\{e\}$, then sending a morphism $g$ of $G$ to the $G$-map $G/\{e\}\to G/\{e\}$, $h\mapsto hg$, defines a functor $i\colon G \to \OFop$.
\begin{lem}
\label{quillenpair}
Let $\F$ be a collection of subgroups of $G$ containing the trivial subgroup $\{e\}$. The precomposition functor $i^*\colon \COFop\to \C^G$ has a fully faithful right adjoint. If $\COFop$ admits the projective model structure and $\C^G$ admits the $\F$-model structure, then $i^*$ is a left Quillen functor.

\end{lem}
\begin{proof}
We describe explicitly the right adjoint $i_*$ of $i^*$. Let $i_*\colon \C^G\to \COFop$ be the functor sending a $G$-object $X$ to the orbit diagram of its fixed point objects. That is $i_* X(G/H)= X^H$ and $i_* X$ applied to some morphism in $\OF$ of the form $R_a\colon G/H\to G/K$ is the map $i_* X(R_a)\colon X^K\to X^H$ induced by the composite $X^K\to X(*)\stackrel{X(a)}{\rightarrow} X(*)$.

Then $i_*$ is right adjoint to $i^*$. Indeed, let $\varepsilon\colon i^*i_* \to \id_{\C^G}$ be the natural transformation given in a $G$-object $X$ by the isomorphism with underlying map $X^{\{e\}}\to X(*)$ in $\C$.

We define a natural transformation $\eta\colon \id_{\COFop}\to i_*i^*$. Let $T$ be an orbit diagram. For an orbit set $G/H$ of $\OFop$, let $(\eta_T)_{G/H}$ be the map from $T(G/H)$ to $(i_*i^*(T))(G/H)=i^*(T)^H$ in $\C$ induced by 
\[T(R_e)\colon T(G/H)\to T(G/\{e\})=i^*(T)(*).\] 
This defines a map $\eta_T\colon T\to i_*i^*(T)$ of orbit diagrams that is natural in $T$.

One checks that $\varepsilon$ is the counit and $\eta$ is the unit of an adjunction $(i^*, i_*)$. Moreover, the right adjoint $i_*$ is fully faithful, since the counit is an isomorphism.

Suppose that $\COFop$ admits the projective and $\C^G$ the $\mathcal F$-model structure. It remains to show that $i_*$ is a right Quillen functor, i.e., that $i_*$ preserves fibrations and acyclic fibrations. But this follows immediately by construction of $i_*$ and the definition of the model category structures. 
\end{proof}

The toy example $\C= 0 \to 1$ equipped with the trivial model structure, where the weak equivalences are the isomorphisms, shows that $\COFop$ and $\C^G$ are not Quillen equivalent in general. For instance for $G$ the cyclic group of order two and $\mathcal F$ the collection of all subgroups, the category of contravariant orbit diagrams has three objects, which are pairwise non-isomorphic, whereas the category of $G$-objects in $\C$ has only two objects.

\begin{thm}
\label{mainthm1}
Let $G$ be a group and $\F$ a collection of subgroups of $G$ containing the trivial subgroup. Suppose that $\C$ is a cofibrantly generated model category and that for any $H\in \F$, the $H$-fixed point functor satisfies the cellularity conditions \ref{cellular}. Then there is a Quillen equivalence
\[
i^*\colon\COFop \leftrightarrows \C^G\colon i_*
\]
between the category of contravariant orbit diagrams with the projective model structure and of $G$-objects with the $\F$-model structure.
\end{thm}

\begin{proof}
The category of $G$-objects $\C^G$ admits the $\mathcal{F}$-model structure by Proposition \ref{propFmodelstructure}. Recall \cite[11.6.1]{hirschhorn} (or deduce from Theorem \ref{transport}) that any category of diagrams in the cofibrantly generated model category $\C$ admits the projective model structure and is again cofibrantly generated. In particular, the category of orbit diagrams $\COFop$ has generating cofibrations 
\[
I_{\OF}=\{\OFop(G/H,-)\otimes f| H\in \mathcal{F}, f\in I\},
\]
where $I$ is a set of generating cofibrations of $\C$. 

Consider the Quillen pair constructed in the proof of Lemma~\ref{quillenpair}. We show that it is a Quillen equivalence, i.e., that for any cofibrant object $T$ of $\COFop$ and fibrant object $X$ in $\C^G$, a morphism $f\colon i^*T\to X$ is a weak equivalence if and only if its adjoint $i_*(f)\eta_T\colon T\to i_*(X)$ is a weak equivalence. By definition of the model category structures, the map $f$ is a weak equivalence if and only if $i_*(f)$ is a weak equivalence. Thus $(i^*,i_*)$ is a Quillen equivalence if and only if the unit $\eta$ is a weak equivalence in every cofibrant $T\in\COFop$. 

We conclude by showing that $\eta_T$ is actually an isomorphism for cofibrant $T$. As any cofibrant $T$ is a retract of an $I_\OF$-cell complex, we can assume that $T$ itself is an $I_\OF$-cell complex. That is, there exists an ordinal $\lambda > 0$ and a $\lambda$-sequence $S\colon \lambda \to \COFop$ with colimit $T$, starting with $S_0$ the initial object and such that for any $\beta + 1 < \lambda$, there is a pushout square

\begin{equation*}
\xymatrix{
\OFop(G/K,-)\otimes A\ar[d]_{\OFop(G/K,-)\otimes f}\ar[r] & S_\beta\ar[d] \\
\OFop(G/K,-)\otimes B \ar[r] & S_{\beta +1}
}
\end{equation*}
for some $K\in \mathcal F$ and some generating cofibration $f$ in $\C$. Note that for any object $C$ of $\C$, the $G$-object $i^*(\OFop(G/K,-)\otimes C)$ is isomorphic to $G/K\otimes C$ by the isomorphism of Lemma~\ref{fixedpointrepresentable}. Moreover, the composite
\[
\OFop(G/K,-)\otimes C\stackrel{\eta}{\rightarrow} i_*i^*(\OFop(G/K,-)\otimes C)\cong i_*(G/K\otimes C) 
\] evaluated in an orbit set $G/H\in\OFop$ agrees with the composite

\[\OFop(G/K,G/H)\otimes C \cong (G/K)^H\otimes C \to (G/K\otimes C)^H
.
\]
Thus the unit is an isomorphism in the orbit diagram $\OFop(G/K,-)\otimes C$ by the cellularity condition \ref{cellular3}. 

By the cellularity conditions on the fixed point functors and since left adjoints preserve colimits, it follows that $i_*i^*T$ is the transfinite composition of pushouts 
\begin{equation*}
\xymatrix{
i_*i^*(\OFop(G/K,-)\otimes A)\ar[d]\ar[r] & i_*i^*S_\beta\ar[d] \\
 i_*i^*(\OFop(G/K,-)\otimes B) \ar[r] & i_*i^*S_{\beta +1}
.}
\end{equation*}
Note that $i_*$ preserves the initial object. Thus $\eta_{S_0}$ is an isomorphism and so is $\eta_T$, by transfinite induction.
\end{proof}

\begin{rem}
\label{rem:mainthm1}
To show that the Quillen pair is a Quillen equivalence, we only used the cellularity condition \ref{cellular3} for domains and codomains $A$ of generating cofibrations.
\end{rem}

\subsection{Positive examples}
\label{sec:positivieexamples}
Let $\mathcal F$ be a collection of subgroups of the group $G$ and let $H\in \mathcal F$. We will list examples of cofibrantly model categories $\C$, where the $H$-fixed point functor $(-)^H\colon \C^G\to \C$ satisfies the cellularity conditions and thus Theorem \ref{mainthm1} applies if $\mathcal F$ contains the trivial subgroup.

We show first how to build new examples out of given ones.
\begin{lem} Suppose that $\C$ is cofibrantly generated and that the $H$-fixed point functor $(-)^H\colon \C^G\to \C$ satisfies the cellularity conditions \ref{cellular}. Then the cellularity conditions are also satisfied for any diagram category $\CD$ with the projective model structure and any left Bousfield localization of $\C$.
\end{lem}
\begin{proof}
For $\C$ cofibrantly generated, any cofibration in $\CD$ with the projective model structure is a cofibration in each level. Thus the first claim follows since colimits and limits in $\CD$ are calculated objectwise. The second claim follows by definition, as a left Bousfield localization of $\C$ is a new model structure on the category $\C$ with the same cofibrations but more weak equivalences.
\end{proof}

 Let $\C_*$ denote the \emph{category of pointed objects in $\C$}, i.e., the category of objects under the terminal object. Let $(-)_+\colon \C\to \C_*$ be the functor that adds a disjoint base point. It is left adjoint to the underlying functor. Recall that $\C_*$ is a model category with cofibrations, weak equivalences and fibrations the maps that are so in $\C$.
\begin{lem}
\label{pointed}
If $\C^G$ admits the $\mathcal F$-model structure, then so does $(\C_*)^G$. If $\C^\OFop$ admits the projective model structure, then so does $(\C_*)^\OFop$. If $\mathcal F$ contains the trivial subgroup and if $i^*\colon\COFop \leftrightarrows \C^G\colon i_*$ is a Quillen equivalence, then so is $i^*\colon(\C_*)^\OFop \leftrightarrows (\C_*)^G\colon i_*$.
\end{lem}
\begin{proof}
The category of $G$-objects in the category of pointed objects $\C_*$ identifies with the category $(\C^G)_*$ of pointed $G$-objects. Under this identification, the $\mathcal F$-model structure on $(\C_*)^G$ corresponds to $(\C^G)_*$ with the underlying model structure from $\C^G$.

Similarly, the projective model structure on $(\C_*)^\OFop$ corresponds to $(\COFop)_*$ with the underlying model structure from $\COFop$.

Recall \cite[1.3.5]{hovey} that any Quillen pair $(F,U)$ induces a Quillen pair between the corresponding model categories of pointed objects. If the left adjoint $F$ preserves the terminal object and if $(F,U)$ is a Quillen equivalence, then the induced Quillen pair is a Quillen equivalence as well. Thus if $i^*\colon\COFop \leftrightarrows \C^G\colon i_*$ is a Quillen equivalence, then so is the induced adjunction $(\COFop)_* \leftrightarrows (\C^G)_*$. We conclude by noting that this induced adjunction identifies with $i^*\colon(\C_*)^\OFop \leftrightarrows (\C_*)^G\colon i_*$.
\end{proof}

We are now ready to see some examples.

Note that the $H$-fixed point functor $\Set^G\to \Set$ preserves directed colimits of diagrams where each arrow is a monomorphisms, preserves pushouts of diagrams where one leg is a monomorphism and for any $G$-sets $X$ and $Y$, the induced map $X^H\times Y^H\to (X\times Y)^H$ is an isomorphism, as limits commute. Hence if $\mathcal F$ contains the trivial subgroup, then Theorem~\ref{mainthm1} applies in the following example and Lemma~\ref{pointed} provides the pointed version.
 
\begin{ex}[Presheaf categories] Let $\C$ be the category of simplicial sets with the Quillen model structure, or any presheaf category with a cofibrantly generated model structure such that the generating cofibrations are monomorphisms. Then every cofibration of $\C$ is a monomorphism and thus the $H$-fixed point functor satisfies the cellularity conditions.
\end{ex}

If $\C$ is the category $\sSet$ of simplicial sets, then the cofibrations of $\sSet^G$ with the $\mathcal F$-model structure can be described explicitly using the following certainly known lemma. We denote the standard $n$-simplex by $\Delta[n]$, its boundary by $\partial \Delta[n]$ and write $\Sk_n\colon \sSet\to \sSet$ for the $n$-skeleton functor with the convention that $\Sk_{-1}X=\emptyset$ for all simplicial sets $X$.
\begin{lem}
\label{cofibrationsSetG}
Let $f\colon A\to B$ be a monomorphism in $\sSet^G$. For $n\geq 0$, denote the $G$-set of non-degenerate $n$-simplices of $B-f(A)$ by $e(B)_n$ and for every orbit of $e(B)_n/G$ choose a representative $x$. Then we obtain a pushout square
\[
\xymatrix{ \coprod_x G/G_x\otimes \partial\Delta[n]\ar[r]\ar[d] & A\cup \Sk_{n-1}B\ar[d] \\
\coprod_x G/G_x\otimes \Delta[n] \ar[r] & A\cup \Sk_nB,
}
\]
 where the lower horizontal map in a coset $gG_x$ is given by the $n$-simplex $gx$ of $\Sk_nB$.
\end{lem}
\begin{proof}
The square is a pushout in $\sSet$ by the non-equivariant analogue as $G/G_x$ identifies with the orbit $Gx$ via $gG_x\mapsto gx$. Since this bijection is equivariant, we also have a pushout square in $\sSet^G$. 
\end{proof}

If $\mathcal F$ consists only of the trivial subgroup, then it is well-known that the cofibrant objects are the simplicial sets with a free $G$-action and more generally by \cite[2.2 (ii)]{drordwyerkan} that the cofibrations of $\sSet^G$ are the monomorphisms $f\colon A\to B$ such that that only the identity element $e\in G$ fixes simplices in $B-f(A)$.

The cofibrations of $\sSet^G$ with the $\mathcal F$-model structure for general $\mathcal F$ are as expected.
\begin{prop}
\label{gssetcofibration}
A map $f\colon A\to B$ in $\sSet^G$ with the $\mathcal F$-model structure is a cofibration if and only if $f$ is a monomorphism and for every $x\in B-f(A)$, the stabilizer $G_x$ is in $\mathcal F$.
\end{prop}
\begin{proof}
Let $\mathcal K$ be the class of monomorphisms $f\colon A\to B$ satisfying the isotropy conditions of the statement.

We show that every cofibration of $\sSet^G$ is in $\mathcal K$. Generating cofibrations 
\[G/H\otimes \partial \Delta[n]\to G/H\otimes \Delta[n]\] 
are in $\mathcal K$, since the stabilizer of $gH\in G/H$ is $gHg^{-1}$. Note that morphisms of $\mathcal K$ are closed under taking pushouts, transfinite composition and retraction. Thus every cofibration is in $\mathcal K$.

Conversely, suppose that $f\colon A\to B$ is in $\mathcal K$. Since $f$ is the transfinite composition of
\[A\to \ldots \to A\cup \Sk_{n-1}B\to A\cup \Sk_nB\to \ldots ,
\]
it follows that $f$ is a cofibration in $\sSet^G$ by Lemma~\ref{cofibrationsSetG}.

\end{proof}

Let $\Gr$ denote the category of groups. For a group $A$ and set $X$, the copower $X\otimes A$ is the free product of copies of $A$ indexed by $X$. Note that 
\[(G/K\otimes A)^H\cong (G/K)^H\otimes A\] 
for any group $A$. Let $F\colon \Set \to \Gr$ denote the free group functor, and let $S\subset T$ be a subset inclusion. Then $(-)^H$ preserves the pushout of any diagram of the form
\[ G/K\otimes FT \leftarrow G/K\otimes FS \rightarrow X
\]
in $\Gr^G$. Indeed, the pushout is the coproduct $G/K\otimes F(T-S)\coprod X$ in $\Gr$. Observe that $(-)^H$ preserves binary coproducts. It follows that $(-)^H$ applied to the pushout yields 
\[((G/K)^H\otimes F(T-S))\coprod X^H\] 
as desired. Since the forgetful functor $\Gr\to \Set$ creates filtered colimits, we deduce that $(-)^H\colon \Gr^G\to \Gr$ preserves directed colimits of diagrams where each arrow is a monomorphism.

\begin{ex}[Simplicial groups]
Let $\C$ be the category of simplicial groups $\sGr$ with Quillen's model structure \cite[II p.~3.7]{quillen}, which is cofibrantly generated, as it can be obtained by transfer from the category of simplicial sets along the free group functor. Moreover, the pushout of a generating cofibration is a monomorphism in $\sGr$, and, as monomorphisms of groups are closed under transfinite composition and monomorphisms are closed under retraction, every cofibration in $\sGr$ is a monomorphism. Thus the cellularity condition \ref{cellular1} is satisfied and so is the cellularity condition \ref{cellular3}. The cellularity condition \ref{cellular2} holds by Proposition \ref{prop:cellular2}. We conclude that the category $\sGr^G$ admits the $\mathcal F$-model structure and is Quillen equivalent to the category of orbit diagrams $\sGr^\OFop$, if $\mathcal F$ contains the trivial subgroup. 
\end{ex}

\begin{ex}[Small categories and posets] If $\C$ is the category of small categories $\Cat$ with the Thomason model structure, then $\Cat^G$ satisfies the cellularity conditions by the work \cite{bohmannetal} of Bohmann et al.~ and thus $\Cat^G$ admits the $\mathcal F$-model structure and is Quillen equivalent to $\Cat^\OFop$, if $\mathcal F$ contains the trivial subgroup. Lemma~\ref{pointed} implies that also the category of equivariant pointed small categories $(\Cat_*)^G$ admits the $\mathcal F$-model structure and that the fixed point diagram functor $(\Cat_*)^G\to (\Cat_*)^\OFop$ is a Quillen equivalence if $\mathcal F$ contains the trivial subgroup.

Similarly, if $\C$ is the category of posets $\Pos$ with Raptis's model structure \cite{raptis}, then $\Pos^G$ satisfies the cellularity conditions by joint work \cite{maystephanzakharevich} with May and Zakharevich. Thus $\Pos^G$ and $(\Pos_*)
^G$ admit the $\mathcal{F}$-model structure and if $\mathcal F$ contains the trivial subgroup, they are Quillen equivalent to $\Pos^\OFop$ and $(\Pos_*)^\OFop$, respectively.
\end{ex}

Bergner applied Theorem~\ref{mainthm1} to various models for $(\infty,1)$-categories in \cite{bergnerequi}.
\subsection{Chain complexes}
\label{sec:chaincomplexes}
Let $R$ be a unitary ring. We show that if $\C$ is the category of chain complexes $\Ch(R)$ of left $R$-modules with the projective model structure, then the model structures on $\Ch(R)^\OFop$ and on $\Ch(R)^G$ exist, but the fixed point diagram functor is not a Quillen equivalence in general. The model structure on the category of orbit diagrams is useful for doing homological algebra of coefficient systems, i.e., of orbit diagrams of $R$-modules. It identifies with the projective model structure on the category of chain complexes in the abelian category of coefficient systems. The $\mathcal F$-model structure on equivariant chain complexes will be used to prove an equivariant Whitehead Theorem for collections.

Let $\mathcal F$ be a collection of subgroups of $G$ and $H\in \mathcal F$. Let $\RMod$ denote the category of left $R$-modules. Then the category of $G$-objects in $\RMod$ is the category of modules over the group ring $R[G]$ and the $H$-fixed points are the $H$-invariants. The cellularity condition \ref{cellular3} does not hold in general. Indeed, for any $R$-module $A$, the fixed point module $(G/K\otimes A)^H$ is $M_K\otimes A$, where $M_K\subset H\setminus(G/K)$ is the subset of finite $H$-orbits. Nevertheless, as for the category of groups, the $H$-fixed point functor preserves pushouts of the form
\[ G/K\otimes FT \leftarrow G/K\otimes FS \rightarrow X,
\]
where $F$ is the free $R$-module functor and $S\subset T$ a subset inclusion, and the $H$-fixed point functor preserves directed colimits of diagrams, where each arrow is a monomorphism.

\begin{ex}[Chain complexes]
\label{chaincomplexes}
Let $\Ch(R)$ be the category of non-negatively graded chain complexes of $R$-modules with the projective model structure \cite[7.2]{dwyerspa}, i.e., the weak equivalences are the quasi-isomorphisms, and the fibrations are the morphisms that are epimorphisms in degrees $n\geq 1$. Write $R\langle n\rangle$ for the chain complex given by $R$ concentrated in degree $n$. Set $D^0=R\langle 0\rangle$ and for $n\geq 1$, let $D^n$ denote the chain complex that contains $R$ in degree $n$ and in degree $n-1$, with differential the identity, and that is zero in the other degrees. Set $S^{-1}=0$ and for $n\geq 0$, let $S^{n}$ denote the chain complex $R\langle n\rangle$. Recall that $\Ch(R)$ is cofibrantly generated with generating cofibrations $\{S^{n-1} \to D^n\}_{n\geq 0}$ and generating acyclic cofibrations $\{0\to D^n\}_{n\geq 1}$. By Remark \ref{rem:cellular} and Proposition \ref{prop:cellular2}, the category $\Ch(R)^G$ admits the $\mathcal F$-model structure, but the Quillen pair $i^*\colon\Ch(R)^\OFop\leftrightarrows\Ch(R)^G\colon i_*$ defined for $\mathcal F$ containing $\{e\}$ is not a Quillen equivalence in general. Indeed, by the proof of Theorem~\ref{mainthm1}, the unit $\eta_T\colon T\to i_*i^* T$ would have to be a weak equivalence for cofibrant $T$, which in general is not the case for $T=\OFop(G/\{e\},-)\otimes R\langle 0\rangle$. For instance, if $G\in \F$ and $G$ is not the trivial group, then $T(G/G)=0$, whereas $(i_*i^* T)(G/G)\cong R\langle 0\rangle$.
\end{ex}

The category $\Ch(R)^\OFop$ identifies with the category of chain complexes in the abelian category of diagrams $\OFop\to \RMod$, i.e., of coefficient systems. Under the above identification, the projective model structure on $\Ch(R)^\OFop$ is just the projective model structure on the category of chain complexes in $\RMod^\OFop$. 

The category $\Ch(R)^G$ identifies with the category of chain complexes of $R[G]$-modules. We use the $\mathcal F$-model structure on $\Ch(R)^G$ and on $\sSet^G$ for a conceptual proof of an equivariant Whitehead Theorem on chain complexes for collections. Kropholler and Wall \cite[1.2]{krophollerwall} showed that an equivariant map $f\colon X\to Y$ between $G$-CW complexes induces an equivariant homotopy equivalence $C(f;R)$ on augmented cellular chain complexes provided that $C(f^H;R)$ is a homotopy equivalence for all subgroups $H$ of $G$. Their proof also works for the unaugmented cellular chain complexes and the equivariant Whitehead Theorem for cellular chain complexes of $G$-CW complexes implies a version for the normalized chain complexes of simplicial $G$-sets. We will provide a conceptual proof of this simplicial set version and generalize it to collections, showing that for $f\colon X\to Y$ in $\sSet^G$ it is enough to check that the map on normalized chain complexes $C(f^H;R)$ is a homotopy equivalence for all subgroups $H$ that appear as isotropy groups of $X$ and $Y$.

Equip the category of $G$-objects in $\Ch(R)$ and the category of $G$-simplicial sets with the $\mathcal F$-model structure. Let $C(-;R)\colon \sSet\to \Ch(R)$ denote the normalized chain complex functor. Since it is a left Quillen functor, so is the induced functor on $G$-objects. Indeed, its right adjoint $\Ch(R)^G\to \sSet^G$ is right Quillen since it commutes with the fixed point functors. Together with Proposition \ref{gssetcofibration}, it follows that $C(X;R)$ is cofibrant under the hypotheses of the following lemma. Moreover, as any object $Z$ of $\Ch(R)^G$ is fibrant, it makes sense to speak about homotopic maps between $C(X;R)$ and $Z$.

\begin{lem} 
\label{chainhomotopic}
Let $X$ be a simplicial $G$-set such that all isotropy groups of simplices of $X$ belong to $\mathcal F$. Let $f,g\colon C(X;R)\to Z$ be maps in $\Ch(R)^G$ from the normalized chain complex $C(X;R)$ to an equivariant chain complex $Z$. If $f$ and $g$ are homotopic in the model category theoretical sense, then $f$ and $g$ are $G$-equivariantly chain homotopic.
\end{lem}
\begin{proof}
Note that the usual composite  $X\coprod X\to X\times \Delta[1]\to X$ is a factorization of the fold map into a cofibration followed by a weak equivalence in $\sSet^G$ with the $\mathcal F$-model structure, i.e., a cylinder object. Application of the left Quillen functor $C(-;R)$ yields a cylinder object $C(X\times \Delta[1];R)$ in $\Ch(R)^G$, as $X$ is cofibrant. By assumption, there exists a homotopy $H\colon C(X\times \Delta[1];R)\to Z$ from $f$ to $g$. Now, consider the composite 
\[
C(X;R)\otimes_RC(\Delta[1];R)\stackrel{\nabla}{\rightarrow} C(X\times \Delta[1];R)\stackrel{H}{\rightarrow} Z
\]
of the Eilenberg-Zilber map \cite[(5.18)]{eilenbergmaclane} with $H$. Recall that the normalized chain complex $C(\Delta[1];R)$ contains $R$ in degree $1$, two copies of $R$ in degree $0$ and is zero in the other degrees. Thus $C(X;R)\otimes_R C(\Delta[1];R)$ in degree $n+1$ identifies with

\[C_{n+1}(X;R)\oplus C_{n+1}(X;R) \oplus C_{n}(X;R).\]

Let $\varphi_{n}\colon C_{n}(X;R)\to Z_{n+1}$ be the restriction of the composite $H\circ \nabla$ to $C_{n}(X;R)$. Then a straightforward computation shows that $\{(-1)^n \varphi_n\}_{n\geq 0}$ is an equivariant chain homotopy between $f$ and $g$.
 \end{proof}

We apply the Whitehead Theorem for model categories \cite[4.24]{dwyerspa} to obtain the following equivariant Whitehead Theorem for collections.
\begin{thm}
\label{whiteheadchain}
Let $G$ be a group and $\mathcal F$ a collection of subgroups of $G$, let $R$ be a unitary ring. Let $f\colon X\to Y$ be a map between simplicial $G$-sets such that all isotropy groups of simplices of $X$ and $Y$ belong to $\mathcal F$. The map $C(f;R)\colon C(X;R)\to C(Y;R)$ between normalized chain complexes is an equivariant homotopy equivalence in each of the following two situations.
\begin{enumerate}[a)]
\item
\label{whiteheadchaina}
 The induced map $C(X;R)^H\to C(Y;R)^H$ is a quasi-isomorphism for all $H\in\mathcal F$.
\item 
\label{whiteheadchainb}
The induced map $C(X^H;R)\to C(Y^H;R)$ is a quasi-isomorphism for all $H\in \mathcal F$.
\end{enumerate}
\end{thm}
\begin{proof}
By assumption, the simplicial $G$-sets $X$ and $Y$ are cofibrant in $\sSet^G$ with the $\mathcal F$-model structure and thus $C(X;R)$ and $C(Y;R)$ are cofibrant in $\Ch(R)^G$ with the $\mathcal F$-model structure. Moreover, the chain complexes $C(X;R)$ and $C(Y;R)$ are also fibrant, since every object in $\Ch(R)^G$ is fibrant.

Thus in situation \ref{whiteheadchaina}, the map $C(f;R)$ is a homotopy equivalence in the model category theoretical sense by the Whitehead Theorem for model categories and thus an equivariant homotopy equivalence by Lemma~\ref{chainhomotopic}.

 In situation \ref{whiteheadchainb}, let $L_R\sSet$ denote the left Bousfield localization \cite[10.2]{bousfieldloc} of the category of simplicial sets with respect to the class of $H_*(-;R)$-isomorphisms. Recall or deduce directly from Bousfield's proof of \cite[10.2]{bousfieldloc} that $L_R\sSet$ is cofibrantly generated. Thus $(L_R\sSet)^G$ admits the $\mathcal F$-model structure. Since the normalized chain complex functor $C(-;R)\colon L_R\sSet \to \Ch(R)$ is a left Quillen functor, the induced functor from $(L_R\sSet)^G$ to $\Ch(R)^G$ with the $\mathcal F$-model structures is again left Quillen. Hence, the functor $C(-;R)$ takes the weak equivalence $f$ between cofibrant objects in $(L_R\sSet)^G$ to a weak equivalence, and the proof reduces to the already shown case \ref{whiteheadchaina}. 
\end{proof}

\section{The topological setting}
\label{sec:topological}
We work henceforth in the category of weak Hausdorff $k$-spaces $\U$ (see \cite[Appendix A]{lewis}). In particular, a map $i\colon X\to Y$ in $\U$ is called an \emph{inclusion} if it is a homeomorphism from $X$ onto its image $i(X)$ with the $k$-subspace topology, i.e., the $k$-ification of the ordinary subspace topology. 

In \S\ref{sec:Topological_model_categories}, we recall the notions of topological model categories and topological categories, i.e., categories  enriched \cite{kelly} in $\U$. Enrichment in $\U$ is particularly simple. For instance $\U$-natural transformations between $\U$-functors are just ordinary natural transformations and thus any adjunction between $\U$-functors is automatically a $\U$-adjunction.

In \S\ref{sec:Equivariant_homotopy_theory_for_topological_model_categories}, the three questions of the introduction regarding the existence of the $\mathcal F$-model structure, of the projective model structure and their comparison are studied in the topological setting. The analogue of Theorem \ref{mainthm1} holds for compact Lie groups $G$. For a general topological group $G$ some extra work, depending on the particular example of topological model category $\C$ considered, is required to show the existence of the projective model structure and of the $\mathcal F$-model structure. We will see in \S\ref{sec:topologicaldiagrams} a positive answer to the three questions for any topological group $G$, when $\C$ is the category of topological diagrams $\UD$ indexed by any small topological category $\D$ with the projective model structure.

\subsection{Topological model categories}
\label{sec:Topological_model_categories}

Let $\U$ denote the cartesian closed category of weak Hausdorff $k$-spaces equipped with the Quillen model structure. We write $\Uu(X,Y)$ for the internal hom of $X$ and $Y$. A \emph{topological model category} $\C$, is a category enriched, tensored and cotensored over the cartesian closed category $\U$, whose underlying category is equipped with a model category structure in a compatible way. This means first of all that $\C$ is a \emph{topological category}: each hom-set $\C(X,Y)$ is topologized as a space $\Cu(X,Y)\in \U$, such that composition is continuous. Moreover, $\C$ is equipped with a functor
$\otimes \colon \U\times \C\to \C$, called \emph{tensor}, and with a functor $\pitchfork\colon \U^\op\times \C\to \C$, called \emph{cotensor}, together with natural isomorphisms 
\[\Cu(X\otimes A,B)\cong \Uu(X,\Cu(A,B))\cong \Cu(A, X\pitchfork B).\] 
Finally, the category $\C$ is equipped with a model category structure such that the \emph{pushout-product axiom} holds: for any cofibration $f\colon X\to Y$ in $\U$ and any cofibration $i\colon A\to B$ in $\C$, the induced map from the pushout $Y\otimes A\cup_{X\otimes A} X\otimes B$ to $Y\otimes B$ is a cofibration in $\C$, which is acyclic if $f$ or $i$ is acyclic.

\begin{ex}[{\cite[4.2.11]{hovey}}]
The category of spaces $\U$ is a topological model category with tensor the cartesian product and cotensor the internal hom $\Uu(-,-)$.
\end{ex}

A functor $F\colon \D \to \E$ between topological categories is a \emph{$\U$-functor} if 
\[\Du(A,B)\to \Eu(FA,FB)\]
 is continuous for all objects $A$ and $B$ of $\D$. 
\begin{ex} If $\D$ and $\E$ are topological categories, then so are $\D^\op$ and $\D\times \E$. The functor $\Du\colon \D^{op}\times \D\to \U$ is a $\U$-functor and if $\D$ is tensored and cotensored, then $\otimes$ and $\pitchfork$ are also $\U$-functors.  
\end{ex} 

For any small, topological category $\D$, we write $\CD$ for the category of \mbox{$\U$-functors} from $\D$ to $\C$. Note that $\CD$ is again a topological category, where the space $\CDu(F,G)$ of natural transformations is a subspace of $\prod_{d\in\D}\Cu(F(d),G(d))$ in $\U$. Moreover, the category $\CD$ is tensored and cotensored with tensor $X\otimes F$ of $X\in \U$ and $F\in \CD$ the composite
$\D\stackrel{F}{\rightarrow} \C \stackrel{X\otimes -}{\rightarrow} \C$
and cotensor $X\pitchfork F$ the composite
$\D\stackrel{F}{\rightarrow} \C \stackrel{X\pitchfork -}{\rightarrow} \C$.

We are interested in the existence of the \emph{projective model structure} on $\CD$, where the weak equivalences and fibrations are defined levelwise. The case $\C=\U$ is due to Piacenza \cite[5.4]{piacenza}. We give a different proof.

\begin{prop} 
\label{projectivemodelstructure}
Consider the category of $\U$-functors $\CD$ from a small topological category $\D$ to a topological model category $\C$. Suppose that $\C$ is cofibrantly generated. The category $\CD$ admits the projective model structure and is a topological model category in each of the following two situations.
\begin{enumerate}[a)]
\item
\label{projective_assumption1}
For any two objects $d,d'$ of $\D$, the functor $\Du(d,d')\otimes -\colon \C\to\C$ preserves cofibrations and acyclic cofibrations.
\item
\label{projective_assumption2}
The model category $\C$ is the category $\U$.
\end{enumerate}
\end{prop}
\begin{proof}
Since $\C$ is cotensored and cocomplete, so is $\CD$. Indeed, for a diagram $F$ in $\CD$, let $\colim F$ be the colimit of $F$ calculated in the category of functors from $\C$ to $\D$. Using that $\C$ is cotensored, one checks that $\colim F$ is a $\U$-functor and thus the colimit of $F$ in $\CD$. Similarly, since $\C$ is tensored and complete, so is $\CD$.

Note that for any object $d$ of $\D$, the evaluation functor $\ev_d\colon \CD\to \C$ is right adjoint to the functor $\Du(d,-)\otimes -$ that sends an object $A$ of $\C$ to the composite
\[\D\stackrel{\Du(d,-)}{\longrightarrow}\U\stackrel{-\otimes A}{\longrightarrow} \C.
\]
In both situations, the projective model structure on $\CD$ is obtained by the Transfer Theorem \ref{transport} applied to the right adjoints $\{\ev_d\}_{d\in\D}$. 

In situation \ref{projective_assumption1}, the assumptions hold by Lemma~\ref{lemtransport1}. Indeed, every evaluation functor preserves all colimits. Moreover, a transfinite composition of pushouts of elements of the form $\Du(d,-)\otimes f$ with $d\in \D$ and $f$ a generating cofibration of $\C$ evaluated in an object $d'\in \D$ is a transfinite composition of pushouts of cofibrations in $\C$ and thus a cofibration itself. The same argument works for acyclic cofibrations. 

In situation \ref{projective_assumption2}, the smallness hypotheses reduce to the fact that every space is small with respect to the inclusions using that closed inclusions are closed under taking products, pushouts and transfinite compositions. The second condition holds by Lemma~\ref{lemtransport2}. Indeed, every object $X$ of $\UD$ is fibrant and the cotensor $I\pitchfork X$ of the unit interval $I\in \U$ and $X$ is a path object.

We have shown the existence of the projective model structure on $\CD$. Finally, the pushout-product axiom is equivalent to the following condition \cite[4.2.2]{hovey}: for any cofibration $i\colon A\to B$ in $\U$ and any fibration $p\colon X\to Y$ in $\CD$, the induced map from the cotensor $B\pitchfork X$ to the pullback $B\pitchfork Y\times_{A\pitchfork Y}A\pitchfork X$ is a fibration in $\CD$, which is acyclic if $f$ or $p$ is acyclic. But this holds by the pushout-product axiom in $\C$. Thus $\CD$ is a topological model category.
 \end{proof}

\subsection{Equivariant homotopy theory for topological model categories}
\label{sec:Equivariant_homotopy_theory_for_topological_model_categories}
Let $\C$ be a topological model category. Let $G$ be a topological group, i.e., a group object in the category of weak Hausdorff $k$-spaces $\U$. We identify $G$ with the topological category with one object $*$ and morphism space $G$. 

\begin{defn}
The category of $G$-objects in $\C$ is the category $\C^
G$ of $\U$-functors from $G$ to $\C$. For a closed subgroup $H$ of $G$, the \emph{$H$-fixed point functor $(-)^H$} is defined as the composition $\C^G\to\C^H\stackrel{\lim}{\rightarrow} \C$ of the restriction with the (conical) limit functor.
\end{defn} 

\begin{rem}
Conical limits are defined in \cite[(3.57)]{kelly} for $\mathcal V$-functors where $\mathcal V$ is any cartesian closed category. For us, it suffices to know that the functor $\lim\colon \C^H\to\C$ takes a $\U$-functor $X$ to the fixed points of the underlying functor $X$ with respect to the underlying discrete group $H$.
\end{rem}

The restriction to closed subgroups is for point-set topological convenience. For instance, the set of orbits $G/H$ with the quotient topology will already be weak Hausdorff. The definition of the $H$-fixed point functor makes sense for any subgroup $H$ of $G$, but $(-)^H$ is naturally isomorphic to $(-)^{\bar{H}}$, where $\bar{H}$ denotes the closure of the set $H$ in $G$. This follows from the special case \cite[\S2.1]{fausk} that $X^H=X^{\bar{H}}$ for $G$-spaces $X$ as the induced map
\[\Map_\C(A,Y^H)\to\Map_\C(A,Y)^H\]
is an isomorphism for any subgroup $H$ and for any objects $A$ of $\C$ and $Y$ of $\C^G$.

\begin{ex}
Recall that the category of weak Hausdorff $k$-spaces $\U$ is a topological model category with tensor the cartesian product and cotensor $\Uu(-,-)$. The category of $G$-objects in $\U$ coincides with the category of $G$-spaces, i.e., of spaces with a continuous $G$-action. Moreover, for any closed subgroup $H$ of $G$, the $H$-fixed point functor takes a $G$-space $X$ to its ordinary fixed point space $X^H=\{x\in X|hx=x\text{ for all } h\in H\}$.
\end{ex} 

We construct explicit left adjoints to the $H$-fixed point functors $\C^G\to \C$ in order to transfer the model category structure from $\C$ to the following model category structure on the category of $G$-objects in $\C$.

\begin{defn}
Let $\F$ be a collection of closed subgroups of $G$. The category $\C^G$ is said to admit the \emph{$\F$-model structure} if it is a model category with weak equivalences and fibrations the maps that are taken to weak equivalences and fibrations in $\C$ by each $H$-fixed point functor with $H\in \F$.
\end{defn} 

As in the discrete case, the following lemma is standard.
\begin{lem} 
\label{topfixedpointrepresentable}
For a closed subgroup $H$ of $G$, consider the homogeneous space $G/H$. Evaluating a $G$-map $G/H\to X$ in the coset $H$ yields an isomorphism $\UGu(G/H,X)\cong X^H$, natural in the $G$-space $X$.
\end{lem}

For any object $A$ of $\C$ and homogeneous space $G/H$ considered as a $G$-space, denote the composition
\[G\stackrel{G/H}{\longrightarrow}\U\stackrel{-\otimes A}{\longrightarrow} \C
\]
by $G/H\otimes A$.
\begin{lem}
The functor $G/H\otimes - \colon \C \to \C^G$ is left adjoint to the $H$-fixed point functor.
\end{lem}
\begin{proof}
The adjunction isomorphism is the composite
\begin{equation*}
\begin{split}
\C^G(G/H\otimes A,X)& \cong \U^G(G/H,\Cu(A,X)) \\
& \cong \U(*,\Cu(A,X)^H) \\
& \cong \U(*,\Cu(A,X^H))\cong \C(A,X^H),
\end{split}
\end{equation*}
where the first isomorphism is induced by the isomorphism expressing that $\C$ is tensored, the second isomorphism is given by Lemma~\ref{topfixedpointrepresentable} and the third one comes from the fact that $\Cu(A,-)$ preserves limits.
\end{proof}

In the discrete case, the tensor was the coproduct and thus $(G/H)^K\otimes -$ preserved cofibrations. In the topological setting, this will be true if $(G/H)^K$ is cofibrant in $\U$, by the pushout-product axiom.

The fact \cite{willson} that $G/H$ and $(G/H)^K$ are smooth manifolds for any compact Lie group $G$ and closed subgroups $H,K$ implies the following lemma.
\begin{lem}
\label{compactlie}
Suppose that $G$ is a compact Lie group. For any closed subgroups $H,K$ of $G$, the spaces $G/H$ and $(G/H)^K$ are cofibrant in $\U$.
\end{lem}

Working with a compact Lie group, we provide two positive answers regarding the existence of the $\F$-model structure.
\begin{prop}
\label{Fmodeltop}
Suppose that $G$ is a compact Lie group and that the topological model category $\C$ is cofibrantly generated. Let $\mathcal F$ be a collection of closed subgroups of $G$. The category of $G$-objects in $\C$ admits the $\mathcal{F}$-model structure and is a topological model category in each of the following two situations.
\begin{enumerate}[a)]
\item
\label{Fmodel_assumption1}
The fixed point functors $(-)^H$ for $H\in\mathcal{F}$ satisfy the cellularity conditions \ref{cellular}.
\item
\label{Fmodel_assumption2}
The cofibrations of $\C$ are monomorphisms, and every object of $\C$ is fibrant.
\end{enumerate}
\end{prop}
\begin{proof}
We apply the Transfer Theorem \ref{transport} to the adjunctions $\{G/H\otimes -, (-)^H\}_{H\in\mathcal F}$. Since $\C$ is cotensored and cocomplete, so is the category of $\U$-functors $\C^G$. Similarly, since $\C$ is tensored and complete, so is $\C^G$.

Note that $(G/H)\otimes -$ as a functor to $\C$ and $(G/H)^K\otimes -$ preserve cofibrations by Lemma~\ref{compactlie} and the pushout-product axiom. Thus, in situation \ref{Fmodel_assumption1}, the conditions of the Transfer Theorem can be checked as in the discrete case \ref{propFmodelstructure}.

In situation \ref{Fmodel_assumption2}, note that every object $X$ of $\C^G$ is fibrant and that the cotensor $I\pitchfork X$ of the unit interval $I$ and the $G$-object $X$ is a path object for $X$. Thus, the second condition of the Transfer Theorem holds by Lemma~\ref{lemtransport2}. 

By abuse of notation, denote the set of generating cofibrations of $\C$ by $I$ and the set of generating acyclic cofibrations by $J$. We show that 
\[I_{\mathcal F}=\{G/H\otimes f|H\in \mathcal{F}, f\in I\}\]
permits the small object argument. Let $A$ be the domain of a generating cofibration $f$ of $\C$. Let $\kappa$ be a cardinal such that $A$ is $\kappa$-small relative to the cofibrations. Let $\lambda\geq \kappa$ be a regular cardinal and $X$ a $\lambda$-sequence of relative $I_{\mathcal{F}}$-cell complexes. Note that the underlying map of a relative $I_{\mathcal{F}}$-cell complex is a cofibration in $\C$, since $G/H$ is cofibrant in $\U$. We show that the induced map
\[\colim_{\beta<\lambda} \C^G(G/H\otimes A, X_\beta)\to \C^G(G/H\otimes A, \colim X)
\]
is a bijection. This is equivalent to showing that the induced map
\[
\colim_{\beta<\lambda} \C(A,X_\beta)^H\to \C(A,\colim X)^H
\]
is a bijection. Since the cofibrations of $\C$ are monomorphisms, the functor $\C(A,X)$ is a directed diagram of monomorphisms of $G$-sets and thus $\colim_{\beta<\lambda} \C(A,X_\beta)^H\cong (\colim_{\beta<\lambda}\C(A,X_\beta))^H$. We conclude by noting that \[\colim_{\beta < \lambda}\C(A,X_\beta)\cong \C(A,\colim X)\]
as $G$-sets by the choice of $\kappa$. Similarly, the set $J_{\mathcal{F}}=\{G/H\otimes f|H\in \mathcal{F}, f\in J\}$ permits the small object argument.

It is left to show that $\C^G$ is a topological model category, i.e., that the pushout-product axiom holds. As in the proof of Proposition \ref{projectivemodelstructure}, one checks the equivalent condition involving the cotensor.   
\end{proof}

The second approach to equivariant homotopy theory for $\C$ is via orbit diagrams.
\begin{defn}
Let $\mathcal F$ be a collection of closed subgroups of $G$. The \emph{orbit category $\OF$ with respect to $\mathcal F$} is the topological category given by the full subcategory of $\U^G$ with objects the homogeneous spaces $G/H$ for $H\in \mathcal{F}$.
\end{defn}

\begin{prop}
\label{projectiveorbit}
Suppose that $G$ is a compact Lie group and that the topological model category $\C$ is cofibrantly generated. Then for any collection of closed subgroups $\mathcal F$ of $G$, the category of $\U$-functors $\COFop$ admits the projective model structure.
\end{prop}
\begin{proof}
The space $\Map_{\OFop}(G/H,G/K)$ is isomorphic to $(G/H)^K$ and thus cofibrant in $\U$ by Lemma~\ref{compactlie}. We conclude by Proposition \ref{projectivemodelstructure}.
\end{proof}

If $\mathcal F$ contains the trivial subgroup, then we can, as in the discrete setting, compare the two approaches using the inclusion functor $i\colon G\to \OFop$.
\begin{lem}
\label{quillenpairtop}
Let $\F$ be a collection of closed subgroups of $G$ containing the trivial subgroup $\{e\}$. The precomposition functor $i^*\colon \COFop\to \C^G$ has a fully faithful right adjoint $i_*$. If $\COFop$ admits the projective model structure and $\C^G$ admits the $\F$-model structure, then $i^*\colon \COFop \to \C^G$ is a left Quillen functor.
\end{lem}
\begin{proof}
The proof is the same as that of its discrete analogue, Lemma~\ref{quillenpair}.
\end{proof}

\begin{rem}
The functors $i^*$ and $i_*$ are $\U$-functors and $(i^*,i_*)$ forms a $\U$-adjunction.
\end{rem}

The adjunction $(i^*,i_*)$ is a Quillen equivalence under a requirement on the fixed point functors.
\begin{prop}
\label{quillenequivalencetop}
Suppose that in the situation of Lemma~\ref{quillenpairtop} the model structures exist and that there exists a set $I$ of cofibrations of $\C$ such that $\COFop$ is cofibrantly generated with generating cofibrations 
\[
\{\OFopu(G/H,-)\otimes f| H\in \mathcal{F}, f\in I\}.
\]
If the fixed point functors satisfy the cellularity conditions \ref{cellular}, then there is a Quillen equivalence $i^*\colon\COFop \leftrightarrows \C^G\colon i_*$. 
\end{prop}
\begin{proof}
The unit of the adjunction $(i^*,i_*)$ is an isomorphism in cofibrant objects as in the proof of Theorem \ref{mainthm1} and a morphism $f$ between $G$-objects is a weak equivalence if and only if $i_*(f)$ is so.
\end{proof}

Combining Proposition \ref{projectiveorbit} and Proposition \ref{Fmodeltop} with the previous proposition yields our main result in the topological setting.

\begin{thm}
\label{mainthm2}
Suppose that $G$ is a compact Lie group and $\F$ a collection of closed subgroups of $G$ containing the trivial subgroup. Suppose that $\C$ is a topological model category that is cofibrantly generated and that for any $H\in \F$, the $H$-fixed point functor satisfies the cellularity conditions \ref{cellular}. Then there is a Quillen equivalence
\[
i^*\colon\COFop \leftrightarrows \C^G\colon i_*
\]
between the category of contravariant orbit diagrams with the projective model structure and of $G$-objects with the $\F$-model structure.
\end{thm}

\subsection{Example: Topological diagrams}
\label{sec:topologicaldiagrams}
Let $\D$ be a small, topological category. Equip the category of $\U$-functors $\UD$ with the cofibrantly generated, topological, projective model structure obtained in Proposition \ref{projectivemodelstructure}. Fix a topological group $G$ in $\U$ and a collection $\mathcal F$ of closed subgroups of $G$ containing the trivial subgroup.

The following lemma is known and checked by hand using that any fixed point set of $X$ in $\UG$ is closed. A pointed analogue is \cite[III Lemma~1.6]{mandellmay}.
\begin{lem}
\label{fixedpointfunctortop}
The $H$-fixed point functor $\U^G\to \U$ preserves directed colimits of diagrams where each arrow is a monomorphism, preserves pushouts of diagrams where one leg is a closed embedding as a map in $\U$ and for any $G$-spaces $X$ and $Y$, the induced map $X^H\times Y^H\to (X\times Y)^H$ is an isomorphism in $\U$.
\end{lem}

For $\C=\UD$ and any topological group $G$ we obtain a positive answer to all three questions of the introduction.
\begin{prop} There is a Quillen equivalence between the category of $\U$-functors $\OFop\to \UD$ with the projective model structure and the category of $G$-objects in $\UD$ with the $\mathcal F$-model structure.
\end{prop}
\begin{proof}
Note that the $H$-fixed point functors for $H\in \mathcal F$ satisfy the cellularity conditions by Lemma~\ref{fixedpointfunctortop}. If $G$ is a compact Lie group, we are done by Theorem \ref{mainthm2}. For general $G$, we apply Proposition~\ref{quillenequivalencetop}. 

Note that the category of $G$-objects in $\UD$ admits the $\mathcal{F}$- model structure. This can be proved for instance as Proposition~\ref{Fmodeltop} in the situation \ref{Fmodel_assumption2} with the following adaption regarding the small object arguments for $I_{\mathcal F}$ and $J_{\mathcal F}$. For general $G$, we can not conclude that the underlying map of a relative $I_{\mathcal F}$-cell complex is a cofibration in $\UD$. Nevertheless, it is a levelwise closed inclusion and every object of $\UD$ is small with respect to the levelwise inclusions. Thus, for a domain $A$ of a generating cofibration of $\UD$, the cardinal $\kappa$ shall be chosen such that $A$ is $\kappa$-small with respect to the levelwise inclusions and similarly for $J_{\mathcal F}$. 

Write $\C=\UD$. Identifying the category of orbit diagrams $\COFop$ with the category of $\U$-functors ${\U}^{\OFop\times \D}$, Proposition~\ref{projectivemodelstructure} implies that $\COFop$ admits the projective model structure with generating cofibrations as desired in Proposition~\ref{quillenequivalencetop}.
\end{proof}

Write $\mathcal{O}_G$ for the orbit category of the collection $\mathcal F$ of all closed subgroups of $G$. Taking $\D$ to be the category with only one arrow, we recover in particular Elmendorf's Theorem~(\cite[VI.~ 6.3]{mayequi}, \cite[V.~ 3.2]{mayequi},\cite{elmendorf}) stating that the homotopy categories $\Ho({\U}^{{\mathcal{O}}_G^\op})$ and $\Ho(\U^G)$ are equivalent.
\begin{cor}[{Piacenza (\cite[VI.~ \S6]{mayequi},\cite{piacenza})}]
There is a Quillen equivalence 
\[{\U}^{{\mathcal{O}}_G^\op}\leftrightarrows \U^G\]
 between the category of topological presheaves indexed by the orbit category $\mathcal{O}_G$ and the category of $G$-spaces.
\end{cor}

\begin{appendix}

\section{Transferring model category structures along a set of left adjoints}
\label{sec:transfer}
An important method to equip a category with a model structure is by transfer along a left adjoint from a cofibrantly generated model category, which we generalize here to transfer along a set of adjoints without passing through products of model categories. First, we recall the terminology of cofibrantly generated model categories in reverse order, using terms that will be introduced subsequently.  

A model category $\C$ is \emph{cofibrantly generated} if there exist sets $I$ and $J$ of morphisms of $\C$ that permit the small object argument, such that the acyclic fibrations are the maps that have the right lifting property with respect to $I$ and the fibrations are the maps that have the right lifting property with respect to $J$. The morphisms of $I$ are called \emph{generating cofibrations} and the morphisms of $J$ are called \emph{generating acyclic cofibrations}. A set $K$ of morphisms in a cocomplete category \emph{permits the small object argument} if the domain of every element of $K$ is small relative to the class of relative $K$-cell complexes. A \emph{relative $K$-cell complex} is a transfinite composition of pushouts of elements of $K$, and we denote the class of relative $K$-cell complexes by $K\text{-cell}$. For a definition of smallness relative to a class of morphisms, we refer to \cite[2.1.3]{hovey}. In our cases, this definition agrees with \cite[10.4.1]{hirschhorn}, as we consider only classes of morphisms that form a subcategory and that are closed under transfinite composition. Hovey's \cite{hovey} and Hirschhorn's \cite{hirschhorn} monographs are excellent introductions to the theory of cofibrantly generated model categories.

It is standard to lift a cofibrantly generated model category structure along a left adjoint \cite{crans}. The Transfer Theorem of Kan \cite[11.3.2]{hirschhorn} works more generally for a set of left adjoints.
\begin{thm}
\label{transport}
Let $\mathcal C$ be a cofibrantly generated model category with generating cofibrations $I$ and generating acyclic cofibrations $J$. Let $\mathcal D$ be a complete, cocomplete category. Given a set of adjunctions $\left\{F_\iota\colon {\mathcal C}\rightleftarrows {\mathcal D}\colon U_\iota\right\}_\iota$, write $FI=\bigcup_\iota\left\{F_\iota(f); f\in I\right\}$ and $FJ=\bigcup_\iota\left\{F_\iota(f); f\in J\right\}$. Suppose that
\begin{enumerate}[i)]
\item
\label{transport1}
the sets $FI$ and $FJ$ permit the small object argument and
\item
\label{Utakes}
 for all $\iota$, the functor $U_\iota$ takes relative $FJ$-cell complexes to weak equivalences.
\end{enumerate}
Then there exists a cofibrantly generated model category structure on $\mathcal D$ with generating cofibrations $FI$, generating acyclic cofibrations $FJ$ and with weak equivalences and fibrations the maps of $\mathcal D$ which by every $U_\iota$ are taken to weak equivalences and fibrations in $\mathcal C$, respectively.
\end{thm}

\begin{proof}
We apply the Recognition Theorem~\cite[2.1.19]{hovey}. The class of morphisms that have the right lifting property with respect to a given class of morphisms $K$ is denoted by $K$-inj. Using the numbering of \cite[2.1.19]{hovey}, the following are the nontrivial points to show.
\begin{enumerate}[4.]
\item
\label{FJcelliswe}
Every relative $FJ$-cell complex is a weak equivalence and has the left lifting property with respect to $FI$-inj.
\item [5. \& 6.]
\label{Iinj=WJinj}
Let $p$ be a map in $\D$. Then $p\in FI\text{-inj}$ if and only if $p$ is a weak equivalence and $p\in FJ\text{-inj}$.
\end{enumerate}
Note that for any $\iota$, any map $i$ in $\C$ and any map $p$ in $\D$, the lifting problems
 \begin{equation*}
\xymatrix{
\Fi(A)\ar[d]_{\Fi(i)}\ar[r] & X\ar[d]^p \\
\Fi(B)\ar[r] & Y
}
\quad\text{and}\quad
\xymatrix{
A\ar[d]_{i}\ar[r] & \Ui(X)\ar[d]^{\Ui(p)} \\
B\ar[r] & \Ui(Y)
}
\end{equation*}
are equivalent by adjointness. Condition 5.\ \& 6.\ follows. Indeed, a map $p$ of $\D$ has the right lifting property with respect to $FI$ if and only if every $U_\iota(p)$ is an acyclic fibration in $\C$. On the other hand, the map $p$ has the right lifting property with respect to $FJ$ and is a weak equivalence if and only if every $U_\iota(p)$ is a fibration in $\C$ and is a weak equivalence in $\C$. 

We verify that condition \ref{FJcelliswe}\ holds. Assumption \ref{Utakes} assures that any $FJ$-cell complex is a weak equivalence in $\D$. Moreover, any $FJ$-cell complex has the left lifting property with respect to $FJ\text{-inj}$ and thus with respect to $FI\text{-inj}$, as we have already shown that $FI\text{-inj}\subset FJ\text{-inj}$.
\end{proof}

Mark W.~Johnson already had the idea of lifting model structures along several adjoints. He and Michele Intermont lift model structures separately and then intersect them \cite[Prop.~8.7]{intermontjohnson}.

\begin{rem}
To answer a private question of Emanuele Dotto, the transfer theorem holds also if one works with a set of adjunctions $\left\{F_\iota\colon {\mathcal C_\iota}\rightleftarrows {\mathcal D}\colon U_\iota\right\}_\iota$, where each $C_\iota$ is a cofibrantly generated model category.
\end{rem}

As for the transfer along one left adjoint, the following two lemmas are used to apply the Transfer Theorem in practice. Recall that for an ordinal $\lambda$, a cocomplete category $\C$ and a class of morphisms $\D$ of $\C$, a \emph{$\lambda$-sequence of maps in $\D$} is a non-empty colimit preserving functor $X\colon \lambda \to \C$ such that for every ordinal $\beta$ with successor $\beta+1< \lambda$, the map $X_\beta\to X_{\beta+1}$ is in $\D$.
\begin{lem}
\label{lemtransport1}
Consider a set of adjunctions $\left\{F_\iota\colon {\mathcal C}\rightleftarrows {\mathcal D}\colon U_\iota\right\}_\iota$ and $FI$, $FJ$ as in the Transfer Theorem~\ref{transport}. Suppose that for each $\iota$, the functor $U_\iota$
\begin{enumerate}[a)]
\item
\label{lemtransport1a}
takes relative $FI$-cell complexes to cofibrations,
\item
\label{lemtransport1b}
takes relative $FJ$-cell complexes to acyclic cofibrations,
\item
\label{lemtransport1c}
preserves the colimit of any $\lambda$-sequence of maps in $FI$-cell and
\item
\label{lemtransport1d}
preserves the colimit of any $\lambda$-sequence of maps in $FJ$-cell,
\end{enumerate}
then the conditions \ref{transport1} and \ref{Utakes} of Theorem~\ref{transport} are satisfied. 
\end{lem}
\begin{proof}
Condition \ref{Utakes} of Theorem~\ref{transport} holds, since acyclic cofibrations are in particular weak equivalences.

Concerning condition \ref{transport1}, we show that $FI$ permits the small object argument. Let $A$ be the domain of a map in $I$. We have to prove that for any $\iota$, the object $F_\iota(A)$ is small relative to $FI$-cell. This means that we have to find a cardinal $\kappa$ such that for every regular cardinal $\lambda\geq \kappa$ and $\lambda$-sequence $X$ of maps in $FI$-cell, the induced map
\[\colim_{\beta<\lambda} \D(F_\iota(A),X_\beta)\to \D(F_\iota(A),\colim X)
\] 
is an isomorphism. Equivalently, the induced map
\[
\colim_{\beta<\lambda} \C(A,U_\iota(X_\beta))\to \C(A,U_\iota(\colim X))
\]
has to be an isomorphism. By assumption c) the functor $U_\iota$ preserves the colimit of $X$ and by c) and a) the composite $U_\iota\circ X\colon\lambda \to \C$ is a $\lambda$-sequence of cofibrations in $\C$. Now, the existence of the desired cardinal $\kappa$ follows from the fact that the domains of generating cofibrations in a cofibrantly generated model category are small with respect to the cofibrations \cite[2.1.16]{hovey}.

Similarly, the set $FJ$ permits the small object argument.
\end{proof}

Adapting Quillen's path object argument \cite[II p.~ 4.9]{quillen} as in \cite[Remark A.4]{schwedeshipley}, \cite[2.6]{bergermoerdijk} yields the following result, which is particularly useful in topological situations.
\begin{lem}
\label{lemtransport2}
Condition \ref{Utakes} of the Transfer Theorem~\ref{transport} holds if $\D$ has path-objects for fibrant objects and a fibrant replacement functor.
\end{lem}
\begin{proof}
A path-object of an object $X$ of $\D$ is a factorization $X\to \Path(X)\to X\times X$ of the diagonal map into a weak equivalence followed by a fibration in $\D$. Let $X\to RX$ denote the weak equivalence in $\D$ obtained by applying the fibrant replacement functor to $X$. 

Let $i\colon A\to B$ be a relative $FJ$-cell complex. Thus $i$ has the left lifting property with respect to $FJ\text{-inj}$, the class of maps in $\D$ that have the right lifting property with respect to $FJ$. In particular, there exists a lift $r\colon B\to RA$ in
\[
\xymatrix{
A\ar[d]_{i}\ar[r] & RA \ar[d] \\
B\ar[r] & \ast
}.
\]
We will show that for any $\iota$, the composite $U_\iota(Ri)\circ U_\iota(r)$ is a weak equivalence. Then, as $U_\iota A\to U_\iota RA$ is a weak equivalence by assumption, the $2$-out-of-$6$ property \cite[9.3]{dhks} applied to
\[
U_\iota A\stackrel{U_\iota(i)}{\rightarrow} U_\iota B \stackrel{U_\iota(r)}{\rightarrow} U_\iota RA \stackrel{U_\iota(Ri)}{\rightarrow}U_\iota(RB)  
\]
implies that $U_\iota(i)$ is a weak equivalence as desired.

Choose a lift $H$ in 
\[
\xymatrix{
A\ar[r]^i\ar[d]_i & B\ar[r] & RB \ar[r] & \Path(RB)\ar[d] \\
B\ar[rr]^{((B\to RB),r)} & &RB\times RA \ar[r]^{\id\times Ri} & RB\times RB
}.
\]
Now, $U_\iota (H)$ is a right homotopy between the weak equivalence $U_\iota B\to U_\iota(RB)$ and the composite $U_\iota(Ri)\circ U_\iota(r)$. Thus the latter map is a weak equivalence as well by the $2$-out-of-$3$ property of weak equivalences applied twice. 
 
\end{proof}

\end{appendix}

\providecommand{\bysame}{\leavevmode\hbox to3em{\hrulefill}\thinspace}
\providecommand{\MR}{\relax\ifhmode\unskip\space\fi MR }
\providecommand{\MRhref}[2]{%
  \href{http://www.ams.org/mathscinet-getitem?mr=#1}{#2}
}
\providecommand{\href}[2]{#2}

\end{document}